\documentclass[10pt]{article}
\usepackage{amsmath, amsfonts, amsthm, amssymb, hyperref, setspace, geometry, authblk, color}
\title{\textbf{On the sum of character codegrees of finite groups}}
\author[1]{Mark L. Lewis}
\author[1]{Quanfu Yan\footnote{Corresponding author.}}

\affil[1]{Department of Mathematical Sciences, Kent State University, Kent, OH 44242, USA}
{
    \makeatletter
    \renewcommand\AB@affilsepx{: \protect\Affilfont}
    \makeatother

    \affil[ ]{\par Email addresses}
    
    \makeatletter
    \renewcommand\AB@affilsepx{, \protect\Affilfont}
    \makeatother

   \affil[ ]{\href{mailto:lewis@math.kent.edu}{lewis@math.kent.edu}}
    \affil[ ]{\href{mailto:qyan5@kent.edu}{qyan5@kent.edu}}

}

\renewcommand{\Affilfont}{\small\it}
\date{}
\usepackage{abstract}

\newtheorem{theorem}{Theorem}[section]
\newtheorem{proposition}[theorem]{Proposition}
\newtheorem{lemma}[theorem]{Lemma}
\newtheorem{corollary}[theorem]{Corollary}
\theoremstyle{definition}

\newtheorem*{question*}{Question A}
\allowdisplaybreaks

\expandafter\let\expandafter\oldproof\csname\string\proof\endcsname
\let\oldendproof\endproof
\renewenvironment{proof}[1][\proofname]{%
  \oldproof[\bfseries\scshape #1]%
}{\oldendproof}

\usepackage{stackengine,scalerel}
\stackMath
\def\trianglelefteqslant{\ThisStyle{\mathrel{%
  \stackinset{r}{.75pt+.15\LMpt}{t}{.1\LMpt}{\rule{.3pt}{1.1\LMex+.2ex}}{\SavedStyle\leqslant}%
}}}
\renewcommand{\unlhd}{\trianglelefteqslant}

\renewcommand{\leq}{\leqslant}
\renewcommand{\geq}{\geqslant}
\begin{document}
\maketitle
\begin{abstract}
\noindent\textbf{Abstract.} Let $\chi$ be an irreducible character of a group $G.$ We denote the sum of the codegrees of the irreducible characters of $G$ by $S_c(G)=\sum_{\chi\in {\rm Irr}(G)}{\rm cod}(\chi).$ We consider the question if $S_c(G)\leq S_c(C_n)$ is true for any finite group $G,$ where $n=|G|$ and $C_n$ is a cyclic group of order $n.$  We show this inequality holds for many classes of groups. In particular, we provide an affirmative answer for any finite group whose order is divisible by up to 99 primes. However, we show that the question does not hold true in all cases, by evidence of a counterexample.

\medskip

\noindent{\bf Keywords:} Irreducible character codegrees, Codegree sum.\\
 \noindent{\bf MSC:} 20C15
\end{abstract}
\section{Introduction}

All groups considered in this paper are finite. Let $G$ be a finite group and $C_n$ be a cyclic group of order $n$. We write ${\rm Irr}(G)$ to denote the set of complex
irreducible characters of $G$. For
$\chi\in {\rm Irr}(G),$ define the number $${\rm cod}(\chi)=\frac{|G:\ker\chi|}{\chi(1)}$$ to be the {\textbf{codegree}} of $\chi.$ This definition of codegree was first introduced by Qian, Wang and Wei in \cite{Qian1}. To see a few of the papers where codegree has been studied, consult \cite{Isaacs2011,Qian2021,Yang2017,Du2016}. 

Recently, a number of papers have considered the relationship between the sum of degrees (denoted by $T(G)$) of the irreducible characters of a group $G$ and the structure of that group. For example, it was proved in \cite{ILM2013} by Isaacs, Loukaki and Moreto that if $T(G)\leq 3k(G)$, $T(G)<\frac{3}{2}k(G)$ or $T(G)< \frac{4}{3}k(G),$ then $G$ is respectively solvable, supersolvable or nilpotent, where $k(G)$ denotes the number of conjugacy classes of $G.$  Motivated by these papers, we consider the sum of the codegrees and want to see what can be sold regarding the sum.
In this paper, we denote the sum of codegrees of irreducible characters of $G$ by $$S_{c}(G)=\sum_{\chi\in {\rm Irr}(G)}{\rm cod}(\chi).$$ 

For an abelian group $G$ of order $n$, it is well known that the set ${\rm Irr}(G)$ forms an abelian group and $G\cong {\rm Irr}(G)$ and the codegree of $\chi\in{\rm Irr}(G)$ is the order of $\chi$ in the group ${\rm Irr}(G).$ It follows that $S_{c}(G)$ is equal to the sum of element orders of $G.$ In 2009, Amiri, Jafarian Amiri and Isaacs proved in \cite{AAIsaacs} that for any noncyclic group $H$, the sum of element orders of $H$ is strictly less than the sum of element orders of the cyclic group of order $|H|$. Hence, it follows that when $G$ is an abelian group of order $n$, then $S_{c}(G)\leq S_{c}(C_n)$ with equality if and only if $G$ is cyclic. 

Inspired by the idea of the result in \cite{AAIsaacs} and the above observation, we consider the following question:

\begin{question*}
Let $G$ be a finite group of order $n$. Is it true that $S_c(G)\leq S_c(C_n),$ where $C_{n}$ is a cyclic group of order $n ?$  
\end{question*}
    
We know the question is true for abelian groups based on the result of Amiri, Jafarian Amiri and Isaacs in \cite{AAIsaacs}. Therefore, the next natural class of groups to look at is nilpotent groups. Perhaps not surprisingly, we can answer the questions for nilpotent groups. 

\begin{theorem}\label{thm1.1}
If $G$ is a nilpotent group of order $n$, then $S_{c}(G)\leq S_{c}(C_{n})$ with equality if and only if $G$ is cyclic.
\end{theorem}

For metacyclic groups that are semidirect products of coprime orders and nonabelian, we get a stronger result as follows.

\begin{theorem}\label{thm1.2}
 Let $G=C_n\rtimes C_m$ be a semidirect product of two cyclic groups $C_n$ and $C_m,$ where $(n,m)=1$ and $C_m$ is not normal in $G.$ Then $S_{c}(G)< \frac{8}{21}S_{c}(C_{nm}).$
\end{theorem}

In fact, if for the largest prime, say $p$, dividing the order of $G,$ a Sylow $p$-subgroup of $G$ is either not normal or nonabelian, we can get the result.

\begin{theorem}\label{thm1.3}
Let $G$ be a group of order $n$ and $P$ a Sylow $p$-subgroup of $G$, where $p$ is the largest prime divisor of the order of  $G.$ If either $P$ is not normal in $G$ or $P$ is nonabelian, then $S_{c}(G)\leq S_{c}(C_{n}).$    
\end{theorem}

Next we discuss how the number of prime divisor of the order of $G$ affects the sum of codegrees. Amazingly, we can show that for groups whose order is divisible by up to $t=99$ distinct primes, the question has a positive answer. 

\begin{theorem}\label{thm1.4}
Let $G$ be a group of order $n$ with $|\pi(G)|=t\leq 99,$ where $\pi(G)$ denotes the set of all primes dividing the order of $G.$ Then $S_{c}(G)\leq S_{c}(C_{n}).$      
\end{theorem}
 
With the above results, one is led to believe that the question will have a positive answer for all finite groups. However, by utilizing the properties of the Euler product and the Riemann zeta function, we construct a family of examples  indicating the question will not always have a positive answer. Thus, we have the following.

\begin{theorem}\label{thm1.5}
There is an integer $n$ and a group $G$ of order $n$ so that $S_{c}(G)> S_{c}(C_{n}).$      
\end{theorem}
 
In Section 2, we present some basic properties of $S_c(G)$. In section 3, we discuss about several classes of groups where the question can be answered positively. We will present the counterexamples in Section 4. 

The work in this paper was completed by the second author (P.h.D student) under the supervision of the first author at Kent State University. The contents of this paper may appear as part of the second author's P.h.D dissertation.

\section{Preliminaries}

In this section, we present some results from the literature and prove some preliminary lemmas. When working with irreducible characters, one needs to understand the conjugacy classes of the group. In particular, we here focus on the number of conjugacy classes and give the following lemmas.

\begin{lemma}\label{lem2.3}
Let $G$ be a non-abelian $p$-group and $k(G)$ be the number of conjugacy classes of $G$.
Then $k(G)<\frac{p+1}{p^{2}}|G|$.
\end{lemma}

\begin{proof}
Let $m$ be the maximal order of ${\rm C}_{G}(x)$ for $x\in G-Z(G)$.
By the class equation for $G$, we have that
$$|G|\geq |Z(G)|+(k(G)-|Z(G)|)\frac{|G|}{m}$$
and then
$$k(G)\leq m+|Z(G)|-\frac{m|Z(G)|}{|G|}
<m+|Z(G)|=|G|(\frac{m}{|G|}+\frac{|Z(G)|}{|G|}).$$
Observe that $|G/Z(G)|\geq p^{2}$ and $\frac{m}{|G|}\leq \frac{1}{p}$ as $G$ is nonabelian.
It follows that $k(G)<\frac{p+1}{p^{2}}|G|$.
\end{proof}

We also need the following observations of Guralnick and Robinson from \cite{Guralnic2006}.

\begin{lemma}\label{lem2.4}{\rm (See \cite[Lemma 2]{Guralnic2006})} Let $G$ be a group and $k(G)$ be the number of conjugacy classes of $G$.

{\rm (1)} For every proper subgroup $H$ of $G,$ we have $\frac{k(G)}{|G|}\leq \frac{k(H)}{|H|}.$

{\rm (2)} Let $P\in {\rm Syl}_p(G),$ where $p\in\pi(G).$ If $P$ is not normal in $G,$ then $\frac{k(G)}{|G|}\leq \frac{1}{p}.$
\end{lemma}

Now we prove the following lemma, which plays an important role while considering nilpotent groups. 

\begin{lemma}\label{lem2.1}
Let $G=H\times K$ be a direct product of $H$ and $K$, then $S_{c}(G)\leq S_{c}(H)\cdot S_{c}(K)$. Furthermore, the equality holds if $(|H|, |K|)=1$.
\end{lemma}

\begin{proof}

 Let $\chi\in {\rm Irr}(H)$ and $\psi\in {\rm Irr}(K)$, then 
$\ker\chi\times \ker\psi\leq \ker(\chi\times \psi)$ and so ${\rm cod}(\chi\times \psi)\leq {\rm cod}(\chi)\cdot {\rm cod}(\psi).$ By {{\cite[Theorem 4.21]{Isaacs1976}}}, the characters $\chi\times \psi$ for $\chi\in {\rm Irr}(H)$ and $\psi\in {\rm Irr}(K)$ are exactly the irreducible characters of $G$. Hence $S_{c}(G)\leq S_{c}(H)\cdot S_{c}(K).$ 

Now assume  $(|H|, |K|)=1$. Let $g=hk\in \ker(\chi\times \psi)$ for some $h\in H$ and $k\in K.$ Then
$$\chi\times \psi(hk)=\chi(h)\psi(k)=\chi(1)\psi(1).$$
It follows that $h\in {Z}(\chi)$ and $k\in {Z}(\psi)$.
Let $\chi_{{Z}(\chi)}=\chi(1)\lambda$ and $\psi_{ Z(\psi)}=\psi(1)\mu$,
where $\lambda$ and $\mu$ are linear characters of $ Z(\chi)$ and $Z(\psi)$, respectively.
So $\lambda(h)=\mu(k)^{-1}$. Let $o(g)$ denote its order for an element $g\in G.$ If either $\lambda(h)\neq 1$ or $\mu(k)\neq 1$, then $o(h)=o(k)$, a contradiction with $(|H|, |K|)=1$.
It follows that $\lambda(h)=\mu(k)=1$ and therefore $h\in \ker{\chi}$ and $k\in \ker\psi$. Hence, $\ker\chi\times \ker\psi=\ker(\chi\times \psi)$, which gives that
$${\rm cod}(\chi\times \psi)=\frac{|H\times K: \ker(\chi\times \psi)|}{\chi\times \psi(1)}
={\rm cod}(\chi)\cdot {\rm cod}(\psi).$$
Thus $S_{c}(G)=S_{c}(H)\cdot S_{c}(K)$.
\end{proof}

The hypothesis $(|H|, |K|)=1$ cannot be dropped from the second statement of Lemma \ref{lem2.1}. For example, $S_{c}(C_{2}\times C_{2})=7$ and $ S_{c}(C_{2})\cdot S_{c}(C_{2})=9>7$. We next compute the sum of codegrees of cyclic groups.

\begin{lemma}\label{lem2.2}
{\rm (1)}  If $G$ is isomorphic to a cyclic group $C_{p^{n}}$ of order $p^{n}$,
where $p$ is a prime and $n$ is a positive integer,
then $S_{c}(G)=\frac{p^{2n+1}+1}{p+1}$.

{\rm (2)} If $G$ is a cyclic group of order $m$,
where $m=p_{1}^{s_{1}}p_{2}^{s_{2}}\cdots p_{t}^{s_{t}}$ and $p_{1}<p_{2}<\cdots<p_{t}=p$,
then $S_{c}(G)=\prod_{i=1}^{t}\frac{p_{i}^{2s_{i}+1}+1}{p_{i}+1}>\prod_{i=1}^{t}\frac{p_i}{p_i+1}m^2\geq \frac{p_{1}}{p+1}m^{2}$.

{\rm (3)} Let $p$ be a prime and $G$ be a $p$-group.
Then $S_{c}(G)\equiv 1\bmod p$.

{\rm (4)} Let $p$ be a prime and $a,b$ be positive integers. Then $S_c(C_{p^a})\cdot S_c(C_{p^{b}})\leq S_c(C_{p^{a+b}}).$

\end{lemma}

\begin{proof}
(1) See \cite[Lemma 2.9]{HLM}.

(2) It follows by Lemma \ref{lem2.1}.
Furthermore, turning $m=p^{n}$ in (2), we will have (1).

(3) Since $G$ is a nilpotent group,
 it follows by \cite{GagolaLewis} that $\chi(1)^{2}$ divides $|G: \ker\chi|$ for every $\chi\in {\rm Irr}(G)$.
Also, ${\rm cod}(\chi)=1$ if and only if $\chi=1_{G}$.
Therefore, $p\mid {\rm cod}(\chi)$ if $\chi\neq 1_{G}$ and we have $S_{c}(G)\equiv 1\bmod p$.

(4) We only need to show that $$\frac{p^{2a+1}+1}{p+1}\cdot\frac{p^{2b+1}+1}{p+1}\leq \frac{p^{2(a+b)+1}+1}{p+1},$$ which is equivalent to $p^{2a}+p^{2b}\leq p^{2a}p^{2b}+1.$ Clearly, $\frac{1}{p^{2a}}+\frac{1}{p^{2b}}\leq 1,$ which implies (4). 
\end{proof}

\section{Groups that satisfy $S_c(G)\leq S_{c}(C_{n})$}

We begin with the proof  of Theorem \ref{thm1.1}, which proves that $S_c(G)\leq S_{c}(C_{n})$ for a nilpotent group $G.$

\begin{proof}[Proof of Theorem \ref{thm1.1}]
In light of Lemma \ref{lem2.1}, it suffices to prove this result for $p$-groups. We are done if $G$ is abelian. Now suppose that $G$ is nonabelian and $|G|=p^n$. Then we have that the derived subgroup $G'$ of $G$ is nontrivial
and so $|G: G'|=p^{a}$ for some integer $1\leq a<n$.
If $\chi\in {\rm Irr}(G)$ is nonlinear,
then $\chi(1)\geq p$ and so ${\rm cod}(\chi)\leq p^{n-1}$.
Now, it follows from Lemma \ref{lem2.3} that
\begin{eqnarray*}
S_{c}(G)&=&\sum_{\chi\in {\rm Lin}(G)}{\rm cod}(\chi)
+\sum_{\chi\in {\rm Irr}(G)-{\rm Lin}(G)}{\rm cod}(\chi)\\
&\leq& S_{c}(G/G')+\sum_{\chi\in {\rm Irr}(G)-{\rm Lin}(G)}p^{n-1}\\
&\leq& S_{c}(C_{p^{a}})+(k(G)-|G: G'|)p^{n-1} \\
&<&\frac{p^{2a+1}+1}{p+1}+(\frac{p+1}{p^{2}}p^{n}-p^{a})p^{n-1}.
\end{eqnarray*}
Notice that $n-a\geq 1$ and so $1-(p+1)p^{n-a-2}\leq 1-(p+1)/p<0$.
Also, $\frac{(p+1)^{2}}{p^{4}}\leq 1$. It follows that
$$p^{2a+1}(1-(p+1)p^{n-a-2})+\frac{(p+1)^{2}}{p^{4}}p^{2n+1}\leq p^{2n+1}$$
and so
$$S_{c}(G)<\frac{p^{2a+1}+1}{p+1}+(\frac{p+1}{p^{2}}p^{n}-p^{a})p^{n-1}\leq
\frac{p^{2n+1}+1}{p+1}=S_{c}(C_{p^{n}}).$$
The proof is complete.
\end{proof}

We now consider coprime nonabelian semidirect products of two cyclic groups. We see that we get a stronger bound in this case. We prove Theorem \ref{thm1.2} as follows.

\begin{proof}[Proof of Theorem \ref{thm1.2}]
Let $A=C_n$, $H=C_m$ and $N=[A,H].$ Notice that $N\leq A$ and $[A/N,HN/N]=[A,H]/N=1.$ It follows that $G/N=A/N\times HN/N$ is abelian, which indicates that $G'\leq N.$ Clearly, $N=[A,H]\leq [G,G]=G'.$ Thus, $G'=[A,H].$  By Fitting's Theorem (see \cite[Theorem 4.34]{Isaacs2008}), we have that $A=C_A(H)\times [A,H]=C_A(H)\times G'.$ Since $C_A(H)\unlhd G,$ the group $G$ can be written as $G=C_A(H)\times (G'\rtimes H).$ By Lemma \ref{lem2.1} $S_c(G)\leq S_c(C_A(H))\cdot S_c(G'\rtimes H).$ Clearly, $1<G'\leq A.$ If $G'<A,$ then $C_A(H)>1$ and $S_c(G'\rtimes H)< \frac{8}{21}S_c(G')\cdot S_c(H)$ by induction. It follows from Lemma \ref{lem2.2} that $S_c(G)<\frac{8}{21} S_c(C_A(H))\cdot S_c(G')\cdot S_c(H) \leq \frac{8}{21}S_c(C_{nm}).$ 

Now we assume that $G'=A$. Then $C_A(H)=1.$ Since $A\cong {\rm Irr}(A),$ we may assume that $A=\langle \lambda \rangle,$ where $\lambda$ is an irreducible character of $A$ and $G$ acts on $A$ by conjugation. 
Notice that $A=\bigcup_{d|n} A_{d}$ is a disjoint union of $A_d$, where $A_d=\{a\in A\ |\ o(a)=d\}$ and $o(a)$ denotes the order of $a.$ It is easy to see that $A_d=\{(\lambda^{\frac{n}{d}})^l\ |\ (l,d)=1, 1\leq l\leq d \}$ and $|A_d|=\psi(d),$ where $\psi(d)$ is the Euler's totient function. Since conjugate elements have the some order, every $G$-orbit on $A$ is contained in some $A_d$ for some $d\ |\ n.$ Now we claim that all $G$-orbits contained in $A_d$ for a fixed $d$ have the same size. Let $I_G(\lambda)$ be the stablilizer of $\lambda$ in the action of $G$ on $A$. Then the size of the orbit containing $\lambda$ is $|G:I_G(\lambda)|.$ Thus we only need to show that $I_G(\lambda^{\frac{n}{d}})=I_G((\lambda^{\frac{n}{d}})^l)$ for any integers $l$ satisfying $ (l,d)=1, 1\leq l\leq d.$ Let $h\in I_G((\lambda^{\frac{n}{d}})^l).$ As $G=A\rtimes H$ and $A$ is cyclic, we may assume that $h$ lies in $H\cap I_G((\lambda^{\frac{n}{d}})^l)$. It follows that $[(\lambda^{\frac{n}{d}})^l]^h=(\lambda^{\frac{n}{d}})^l=[(\lambda^{\frac{n}{d}})^h]^l$ and so $[(\lambda^{\frac{n}{d}})^h(\lambda^{\frac{n}{d}})^{-1}]^l=1.$ Since $\langle \lambda^{\frac{n}{d}}\rangle\unlhd G$ and $|\langle \lambda^{\frac{n}{d}}\rangle|=d$, we have that the order of $(\lambda^{\frac{n}{d}})^h(\lambda^{\frac{n}{d}})^{-1}$ divides $d.$ Notice that $(l,d)=1.$ Therefore $(\lambda^{\frac{n}{d}})^h=\lambda^{\frac{n}{d}}$ and so $h\in I_G(\lambda^{\frac{n}{d}})$, as claimed. 

Let $O_{d1}, O_{d1},..., O_{dr_d}$ be all $G$-orbits contained in $A_d$ and $\lambda_{di}$ be a representative in $O_{di}.$ We write $T_{d}=I_G(\lambda^{\frac{n}{d}})$ and $T_{di}=I_G(\lambda_{di}).$ By the above claim, we have that $T_d=T_{di}$ for all $1\leq i\leq r_d.$ Notice that $\lambda_{di}$ is linear and $T_{di}$ splits over $A$. Then $\lambda_{di}$ extends to $T_{di}$ and by Gallagher's theorem \cite[Corollary 6.17]{Isaacs1976} and the Clifford correspondence \cite[Theorem 6.11]{Isaacs1976}, there is a bijection from ${\rm Irr}(T_{di}/A)\to {\rm Irr}(G|\lambda_{di})$ given by $\phi\mapsto (\beta_{di}\phi)^G,$ where $\beta_{di}\in {\rm Irr}(T_{di})$ and $(\beta_{di})_A=\lambda_{di}.$ Next we show that ${\rm cod}(\beta_{di}\phi)^G\leq {\rm cod}\lambda_{di}\cdot{\rm cod}\phi=d\cdot{\rm cod}\phi.$ The last equality holds because ${\rm cod}\lambda_{di}=o(\lambda_{di})$ and $\lambda_{di}\in A_d.$ Since ${\rm ker}\lambda_{di}\ {\rm char}\ A\unlhd G,$ we have that ${\rm ker}\lambda_{di}\unlhd G$ and so ${\rm ker}\lambda_{di}\leq {\rm ker}(\beta_{di}\phi)^G.$ Let $h\in {\rm ker}\phi \cap T_{di}\cap H$ and $g=ya\in G$ with $y\in H$ and $a\in A.$ As $H$ is cyclic, we have that $h^g=h^a.$ Notice that $\beta_{di}\phi(h^a)=\beta_{di}\phi(a^{-1}ha)=\beta_{di}\phi(h)=\beta_{di}(h)\phi(h)=\lambda_{di}(1)\phi(1)=1.$ Hence, $h^g\in {\rm ker}\beta_{di}\phi$ for all $g\in G.$ It follows that $h\in {\rm ker}(\beta_{di}\phi)^G=\cap_{g\in G}({\rm ker}\beta_{di}\phi)^g$ and thus ${\rm ker}(\beta_{di}\phi)^G\geq {\rm ker}\lambda_{di}\cdot({\rm ker}\phi \cap T_d\cap H).$ Hence, 
\begin{eqnarray*}
{\rm cod}(\beta_{di}\phi)^G&=&\frac{|G:T_{di}||T_{di}|}{|G:T_{di}|\lambda_{di}(1)\phi(1)|{\rm ker}(\beta_{di}\phi)^G|}\\
&\leq&\frac{|A||T_{di}\cap H|}{\lambda_{di}(1)\phi(1)|{\rm ker}\lambda_{di}||({\rm ker}\phi \cap T_{di}\cap H)|}\\
&=& {\rm cod}\lambda_{di}\cdot{\rm cod}\phi\\
&=& d\cdot{\rm cod}\phi.
\end{eqnarray*}
Note that by \cite[Theorem 13.1]{Isaacs1976} we have that $|{\rm Irr}_H(A)|=|{\rm Irr}(C_A(H))|=1,$ where ${\rm Irr}_H(A)=\{\chi\in {\rm Irr}(A)| \chi^s=\chi\ {\rm for\ all}\ s\in H \}.$ This implies that $I_G(\lambda)<G$ for $1_A\not= \lambda\in A.$ Let $p$ and $q$ be the smallest prime divisors of $n$ and $m,$ respectively. Hence for $d\not=1,$ $|O_{di}|=|G:T_{di}|\geq q$ and $S_c(T_d/A)\leq \frac{1}{S_c(C_q)}S_c(G/A)=\frac{1}{q(q-1)+1}S_c(H).$ On the other hand, obviously $p(p-1)+1=S_c(C_p)\leq S_c(A).$ Notice that ${\rm Irr}(G)$ the disjoint union of the sets ${\rm Irr}(G|\lambda_{di})$ for $d|n$ and $1\leq i\leq r_d.$ Together with the facts that $T_{di}=T_d$ and $r_d=\frac{|A_d|}{|O_{di}|}\leq \frac{\phi(d)}{q},$ we have 
\begin{eqnarray*}
S_{c}(G)&=&\sum_{\chi\in {\rm Irr}(G)}{\rm cod}(\chi)\\
&=&\sum_{d|n}\sum_{i=1}^{r_d}\sum_{\phi\in{\rm Irr}(T_{di}/A)}{\rm cod}(\beta_{di}\phi)^G\\
&\leq&\sum_{d|n}\sum_{i=1}^{r_d}\sum_{\phi\in{\rm Irr}(T_{d}/A)}d\cdot{\rm cod}\phi\\
&=&\sum_{d|n} d\cdot r_d\cdot S_c(T_d/A)\\
&\leq& S_c(H)+
\frac{1}{q(q-1)+1}\sum_{1<d|n} d\cdot \psi(d)\cdot S_c(H)\\
&\leq& \frac{1}{S_c(A)}S_c(A)\cdot S_c(H)+\frac{1}{q}\frac{1}{q(q-1)+1}(S_c(A)-1)\cdot S_c(H)\\
&<&(\frac{1}{p(p-1)+1}+\frac{1}{q}\frac{1}{q(q-1)+1}) S_c(A)\cdot S_c(H)\\
&\leq& (\frac{1}{3}+\frac{1}{3}\frac{1}{7})S_c(A)\cdot S_c(H)=\frac{8}{21}S_c(A)\cdot S_c(H).
\end{eqnarray*} 
The last inequality follows while taking $p=2$ and $q=3.$ Now the proof is complete.
\end{proof}

We do not know if this is the best bound for groups $G$ as described in Theorem \ref{thm1.2}. Notice that $S_c(S_3)=6$ and $S_c(C_2\times C_3)=21.$ Thus it is reasonable to guess that $S_c(G)\leq \frac{2}{7}S_c(C_{nm})$, with equality if and only if $n=2,m=3$ and $G=S_3$ is the symmetric group of order $6.$

We now prove that if $p$ is the largest prime divisor of the order of $G$ and $G$ has a Sylow $p$-subgroup that is either not normal or nonabelian, then $G$ satisfy the bound. Thus, for groups that do not satisfy the bound, the Sylow subgroup for the largest prime divisor must be normal and abelian. Below we provide the proof of Theorem \ref{thm1.3}.

\begin{proof}[Proof of Theorem \ref{thm1.3}]  Let $n=p_{1}^{s_{1}}p_{2}^{s_{2}}\cdots p_{t}^{s_{t}}$ with $t>1$ and $p_{1}<p_{2}<\cdots<p_{t}=p.$ Assume the theorem is not true. Then $$\frac{S_c(G)}{k(G)}>\frac{S_c(C_n)}{k(G)}\geq \prod_{i=1}^{t}\frac{p_i}{p_i+1}\frac{n}{k(G)}n >\frac{p_1}{p+1}\frac{n}{k(G)}n.$$
If $P$ is not normal in $G$, the by Lemma \ref{lem2.4}(2), we have that $\frac{n}{k(G)}\geq p.$ If $P$ is nonabelian, it follows from Lemma \ref{lem2.3} and  \ref{lem2.4}(1) that $\frac{n}{k(G)}\geq \frac{|P|}{k(P)}>\frac{p^2}{p+1}.$ Hence, $\frac{n}{k(G)}>\frac{p^2}{p+1}$ in both cases. Notice that $\frac{p_1}{p+1}\frac{p^2}{p+1}\geq 2\frac{3^2}{4^2}=\frac{9}{8}>1$ as $p>p_1\geq 2.$ Therefore $\frac{S_c(G)}{k(G)}>n,$ contrary to the fact that ${\rm cod}\chi\leq n$ for all $\chi\in{\rm Irr}(G).$ 
\end{proof}

A similar proof works for the second largest primes dividing the order of $G,$ but it seems redundant to include it. In addition, from the above proof, we can see that the number of prime divisors of the order of $G$ has an important impact on the structure of Sylow subgroups. Hence, it might be interesting to determine how many prime divisors $G$ has to have normal and abelian Sylow subgroups in examples that do not satisfy the bound.

We now turn to determining how many primes dividing the order of the group force the group to satisfy the bound. Applying Theorem \ref{thm1.3}, we first give the following proposition. 

\begin{proposition}\label{prop3.2} Let $G$ be a group of order $n=p_{1}^{s_{1}}p_{2}^{s_{2}}\cdots p_{t}^{s_{t}}$ with $p_{1}<p_{2}<\cdots<p_{t}.$ If $\prod_{i=1}^{t}\frac{p_i}{p_i+1}\geq \frac{7}{48},$ then $S_c(G)\leq S_c(C_n).$
\end{proposition}

\begin{proof} There is nothing to show if $G$ is abelian. Suppose that $G$ is nonabelian. Let $P$ be a Sylow $p_t$-subgroup of $G$. Then by Theorem \ref{thm1.1} and \ref{thm1.3}, we only need to consider the case when $t>1$ and $P$ is abelian and normal in $G.$ The Schur–Zassenhaus theorem indicates that $G$ can be written as $G=P\rtimes H$ for some subgroup $H$ of $G.$ Let $d=\min\{\chi(1)|\chi(1)>1\}.$ Then $d\mid |G:P|$ by \cite[Theorem 6.15]{Isaacs1976} and so $(d,p_t)=1.$ Clearly, $d\geq p_1.$ If $|G'|=p_1,$ then $G'\leq Z(G)$ and thus $\chi(1)^2=|G:Z(G)|$ for all nonlinear characters $\chi\in{\rm Irr}(G).$ This yields $d^2=|G:Z(G)|$ and $p_t\nmid|G:Z(G)|.$ It follows that $P\leq Z(G)$ and hence $G=P\times H.$ By induction and Lemma \ref{lem2.1}, $S_c(G)=S_c(P)\cdot S_c(H)\leq S_c(C_n).$ 

We now assume that $m=|G'|>p_1.$ It is obvious that $m\geq 3$ and $S_c(G/G')\leq \frac{1}{S_c(C_m)}S_c(C_n)\leq \frac{1}{7}S_c(C_n).$
Notice that $|G|=|G:G'|+\sum_{\chi(1)>1}\chi(1)^2\geq |G:G'|+(k(G)-|G:G'|)\cdot d^2.$ Then $k(G)-|G:G'|\leq \frac{|G|}{d^2}.$ Since ${\rm cod}\chi\leq \frac{|G|}{d}$ for all nonlinear characters $\chi\in{\rm Irr}(G),$ we have
\begin{eqnarray*}
S_{c}(G)&=&S_c(G/G')+\sum_{\chi\in {\rm Irr}(G)-{\rm Lin}(G)}{\rm cod}(\chi)\\
&\leq& \frac{1}{7}S_c(C_n) + \frac{n^2}{d^3}\\
&<& \frac{1}{7}S_c(C_n) +\frac{1}{8}\frac{1}{\prod_{i=1}^{t}\frac{p_i}{p_i+1}} S_c(C_n)\\
&\leq&(\frac{1}{7}+\frac{1}{8}\frac{48}{7})S_c(C_n)=S_c(C_n).
\end{eqnarray*} 
The third inequality follows from Lemma \ref{lem2.2}(2). The proof is complete. 
\end{proof}

The above theorem implies that the number $\prod_{i=1}^{t}\frac{p_i}{p_i+1}$ has an important effect on the codegree sum, and $\prod_{i=1}^{t}\frac{p_i}{p_i+1}\to 0$ as $t\mapsto \infty.$ If $p_{1},p_{2},...,p_{t}$ are the first $t$ prime numbers, then $\prod_{i=1}^{t}\frac{p_i}{p_i+1}\geq \frac{7}{48}$ for $t\leq99,$ which hence gives the following corollary.  

\begin{corollary}\label{cor2.7} Let $G$ be a group of order $n$ with $|\pi(G)|\leq 99.$ Then $S_c(G)\leq S_c(C_{n}).$
\end{corollary}

Since the proof of Proposition \ref{prop3.2} provides a rough estimation, it is probable that 99 is not the optimal estimate for the value of $|\pi(G)|$ such that $S_c(G)\leq S_c(C_{n}).$ Therefore, it might be interesting to determine the largest number of $t$ so that $S_c(G)\leq S_c(C_{n})$ for any group $G$ satisfying $|\pi(G)|\leq t.$

Next we consider the semidirect product $G=A\rtimes S$ of a subgroup $A$ and a subgroup $S.$ However, additional hypothesis are needed for the group $G$ to satisfy the bound. Before presenting the result and its proof, we introduce the following fact {\rm (see \cite[Theorem 9]{Guralnic2006})}: $\frac{k(G)}{|G|}\leq \sqrt\frac{1}{|G:{\rm sol}(G)|}$ with equality if and only if $G$ is abelian, where ${\rm sol}(G)$ is the solvable radical of $G$.

\begin{proposition}\label{prop2.8} Let $G=A\rtimes S$ be the semidirect product of a subgroup $A$ and a subgroup $S$ with ${\rm sol}(S)=1.$ Suppose that $(|A|,|S|)=1$ and $|\pi(A)|\leq|\pi(S)|$ and $p>q$ for any primes $p\in \pi(A)$ and $q\in\pi(S).$ Then $S_{c}(G)\leq S_c(C_{n}),$ where $n=|G|.$
\end{proposition}

\begin{proof} Suppose that $|A|=\prod_{i=1}^{t}p_i^{s_i}$ and $|S|=\prod_{i=1}^{l}q_i^{r_i}.$ Then by the above fact and Lemma \ref{lem2.4}(1), we have $\frac{k(G)}{|G|}\leq \frac{k(S)}{|S|}\leq \prod_{i=1}^{l}q_i^{-\frac{r_i}{2}}.$ Assume that the proposition is false. Then by Lemma \ref{lem2.2}(2), we have $$\frac{S_c(G)}{k(G)}>\frac{S_c(C_{n})}{k(G)}>\prod_{i=1}^{t}\frac{p_i}{p_i+1}\prod_{i=1}^{l}\frac{q_i}{q_i+1}\frac{|G|}{k(G)}|G|.$$
We claim that $\frac{q_i}{q_i+1}q_i^{\frac{r_i}{2}}\geq 1+\frac{1}{q_i+1}$ for $1\leq i\leq l.$ If $q_i=2$ for some $i,$ then $r_i\geq 2,$ otherwise $S$ has a normal subgroup of odd order, contrary to ${\rm sol}(S)=1.$ Thus $\frac{q_i}{q_i+1}q_i^{\frac{r_i}{2}}\geq \frac{2}{3}\cdot2=1+\frac{1}{2+1},$ as claimed. If $q_i>2,$ then $q_i\cdot q_i^{\frac{r_i}{2}}-q_i-2\geq q_i^{\frac{3}{2}}-q_i-2\geq 0.$ Hence, the claim holds. Since $t\leq l$ and $p_i\geq q_i+1$ for all $1\leq i\leq t,$ we have that  $\frac{p_i}{p_i+1}\frac{q_i+2}{q_i+1}\geq \frac{q_i+1}{(q_i+1)+1}\frac{q_i+2}{q_i+1}=1.$ It follows that $$\frac{S_c(G)}{k(G)}>\prod_{i=t+1}^{l}(1+\frac{1}{q_i+1})|G|\geq |G|.$$
This is a contradiction. 
\end{proof}

\section{A family of examples to Question A.}

In this section, we begin with some well-known concepts and results. Let $\mathbb{P}$ be the set of all primes. We write $\zeta(s)$ to denote the Riemann zeta function, which is defined as $\zeta(s)=\sum_{n=1}^\infty\frac{1}{n^s}$ for Re$(s)>1.$ The following are some basic properties of several Euler products attached to $\zeta(s).$
Theorem 280 in \cite{Hardy2008} shows that if $s>1,$ then $\zeta(s)$ is convergent and $$\zeta(s)=\prod_{p\in\mathbb{P}}\frac{1}{1-p^{-s}}.$$ 
It follows immediately from $\prod_{p\in\mathbb{P}}(1+\frac{1}{p^{s}})\cdot \prod_{p\in\mathbb{P}}(1-\frac{1}{p^{s}})=\prod_{p\in\mathbb{P}}(1-p^{-2s})$ that

$$\prod_{p\in\mathbb{P}}(1+\frac{1}{p^{s}})=\frac{\zeta(s)}{\zeta(2s)}.$$ 

The construction of examples also relies on the Dirichlet’s theorem on arithmetic progressions (see \cite[Theorem 15]{Hardy2008}), which states that: If $a$ and $b$ are positive integers and $(a,b)=1$, then there are infinitely many primes in the arithmetic progression $\{ak+ b|k\ {\rm is\ an\ integer}\}$ and the sum of the reciprocals of the prime numbers in the progression diverges. In other words, 
 $$\sum_{p\equiv b({\rm mod}\ a)} \frac{1}{p}=\infty.$$

In addition, to demonstrate the existence of the group $G$ we construct, we need a short fact: A semidirect product $C_p^n\rtimes C_q$ exists for distinct primes $p$ and $q$ if and only if $q|p^m-1$ for some $m\leq n.$ Now we have all components to construct examples. To enhance readability, we present the family of examples as follows. 

\begin{theorem}\label{thm4}
Let $G=A\rtimes P$ be a semidirect product of $A$ and $P$, where $P\cong C_3$ and $A\cong C_{p_{1}}^{2}\times C_{p_{2}}^{2}\times \cdot\cdot\cdot \times C_{p_{t}}^{2}$ and $p_1, p_2,...,p_t$ are the first $t$ prime numbers congruent to $2$ modulo $3$. Suppose that $C_A(P)=1.$
Then $$\frac{S_c(G)}{S_c(C_{n})}\to \infty\ \ \ {\rm as}\ \ \ t \to \infty,$$ 
where $n$ is the order of $G.$
\end{theorem}

\begin{proof}[Proof of Theorem \ref{thm4}] Notice that $p_i\equiv 2({\rm mod}\ 3)$ and so $3\mid p_i^2-1$ and $3\nmid p_i-1$ for $1\leq i\leq t.$ By the above fact, the group $C_{p_i}^2\rtimes P$ with $C_{C_{p_i}^2}(P)=1$ exists and hence such a group $G$ exists. In addition, it is easy to see that $G$ is a Frobenius group with kernel $A$. 
Let $G$ act on ${\rm Irr}(A)$ and $1_A=\lambda_0,\lambda_1,...,\lambda_r$ be the set of representatives for the $G$-orbits on ${\rm Irr}(A).$ Then we have that $I_G(\lambda_j)=A$ for $1\leq j\leq r.$ It follows that the size of each non trivial $G$-orbit is $|G:I_G(\lambda_j)|=|P|.$ Obviously, $\lambda_j^G\in {\rm Irr}(G)$ and $${\rm Irr}(G)={\rm Irr}(G/A)\bigcup(\bigcup_{j=1}^r\{\lambda_j^G\}).$$
Notice that if $\phi$ lies in the $G$-orbit containing $\lambda,$ then ${\rm cod}(\phi^G)={\rm cod}(\lambda^G).$ Therefore, 
\begin{eqnarray*}
S_{c}(G)&=&\sum_{\chi\in {\rm Irr}(G)}{\rm cod}(\chi)\\
&=& S_c(G/A)+\sum_{j=1}^r{\rm cod}(\lambda_j^G)\\
&=& S_c(P)+\frac{1}{|P|}\sum_{1_A\not=\lambda\in{\rm Irr}(A)}{\rm cod}(\lambda^G)\\
&=& S_c(P)+\frac{1}{|P|}\sum_{i=1}^t\sum_{|\pi(o(\lambda))|=i}{\rm cod}(\lambda^G).
\end{eqnarray*}

Notice that ${\rm cod}(\lambda^G)=\frac{|A|}{|{\rm ker}\lambda^G|}.$ To compute the codegree sum, we need the following facts. Since the order of $\lambda$ is $o(\lambda)=|A:{\rm ker}\lambda|$ and $\lambda\in {\rm Irr}(A)\cong A,$ it follows that $p_i^2\nmid o(\lambda)$ for any $i.$
If $p_i\mid o(\lambda),$ then $|{\rm ker}\lambda|_{p_i}=p_i.$ Since $C_A(P)=1$ and $3\nmid p_i-1,$ we have that $|{\rm ker}\lambda^G|_{p_i}=1$ and thus $p_i^2\mid{\rm cod}(\lambda^G).$ If $p_i\nmid o(\lambda),$ then $|{\rm ker}\lambda|_{p_i}=p_i^2$ and so $|{\rm ker}\lambda^G|_{p_i}=p_i^2$ and so $p_i\nmid{\rm cod}(\lambda^G).$ We here consider all irreducible characters $\lambda$ satisfying $|\pi(o(\lambda))|=t$ as an example.
Clearly, there are $\prod_{i=1}^t(p_i^2-1)$ such characters and ${\rm cod}(\lambda^G)=\prod_{i=1}^tp_i^2.$ It follows that $$\sum_{|\pi(o(\lambda))|=t}{\rm cod}(\lambda^G)=\prod_{i=1}^tp_i^2(p_i^2-1).$$
With the same idea, we have
\begin{eqnarray*}
S_{c}(G)&=&S_c(P)+\frac{1}{|P|}(\prod_{i=1}^t[p_i^2(p_i^2-1)+1]-1)\\
&=& 7-\frac{1}{3}+\frac{1}{3}\prod_{i=1}^t[p_i^2(p_i^2-1)+1]\\
&=& \frac{20}{3}+\frac{1}{3}\prod_{i=1}^t\frac{p_i^6+1}{p_i^2+1}.
\end{eqnarray*}

By Lemma \ref{lem2.2}(2), we have that $S_c(C_{n})=7\prod_{i=1}^t\frac{p_i^5+1}{p_i+1}.$ To complete our proof, we only need to show that $$r=\prod_{p\equiv 2({\rm mod}\ 3)}\frac{(p^6+1)(p+1)}{(p^2+1)(p^5+1)}=\infty.$$
Indeed, by applying the properties of the above Euler products and Dirichlet’s theorem, we have
\begin{eqnarray*}
\prod_{p\equiv 2({\rm mod}\ 3)}\frac{(p^6+1)(p+1)}{(p^5+1)(p^2+1)} &=& \prod_{p\equiv 2({\rm mod}\ 3)}\frac{(1+\frac{1}{p^6})(1+\frac{1}{p})}{(1+\frac{1}{p^5})(1+\frac{1}{p^2})}\\
 &\geq& \frac{\prod_{p\equiv 2({\rm mod}\ 3)}(1+\frac{1}{p})}{\prod_{p\in\mathbb{P}}(1+\frac{1}{p^5})\prod_{p\in\mathbb{P}}(1+\frac{1}{p^2})}\\
 &\geq&\frac{\sum_{p\equiv 2({\rm mod}\ 3)} \frac{1}{p}}{\frac{\zeta(5)}{\zeta(10)}\frac{\zeta(2)}{\zeta(4)}}=\infty. 
\end{eqnarray*}
\end{proof}

Theorem \ref{thm1.5} follows immediately. Unfortunately we are not able to determine the smallest number $t=|\pi(A)|$ of prime divisors of $A$. Actually, based on the proof of the above theorem, in order to show that $S_c(G)>S_c(C_n),$ it suffices to prove that $r\geq 21.$ We computed all primes less than $1$ billion satisfying $p\equiv 2({\rm mod}\ 3)$ and that the ratio $r$ is approximately to $4.$ We also noticed that  $$\sum_{p\leq m, p\equiv 2({\rm mod}\ 3)} \frac{1}{p}\approx \frac{\log(\log(m))}{2},$$ 
and that the sum diverges very very slowly. Hence, from a certain perspective, this value $t$ would be pretty large. Again, as we mentioned before, it might be interesting to determine the smallest possible number of prime divisors of a group that do not satisfy the bound.

Notice that the group $G$ described in the above theorem is solvable. We also are able to construct a non-solvable group $\Gamma$ such that $S_c(\Gamma)>S_c(C_{|\Gamma|}).$ Let $G$ have the same structure in Theorem \ref{thm1.5} but $2,5\not\in\pi(G).$ Consider the group $\Gamma=G\times Sz(8),$ where $Sz(8)$ is a Suzuki group of order $2^6\cdot5\cdot7\cdot13.$ Clearly, $(|G|,|Sz(8)|)=1.$ This yields that $S_c(\Gamma)=S_c(G)\cdot S_c(Sz(8))$ by Lemma \ref{lem2.1}. Theorem \ref{thm1.5} indicates that we can choose an integer $t$ such that $S_c(\Gamma)>S_c(C_{|\Gamma|}).$

On the other hand, it is clear that the group $G$ constructed in the above theorem has Fitting height $2.$ It is natural to ask: Is there any group $G$ having Fitting height $l$ for any positive integer $l\geq 2$ so that it does not satisfy the bound? Applying the above theorem, we are able to show that such groups exist. Indeed, we first choose a group $H$ of order $h$ so that it has Fitting height $l\geq 2$ and $(3,h)=1.$ (One way to obtain such a group is to take wreath product of cyclic groups, alternating between ones of order $p$ and ones of order $q,$ where $p$ and $q$ are two distinct primes.) Now let $G=A\rtimes P$ be of order $n$ as in Theorem \ref{thm4}. But those prime divisors of the order of $A$ can be carefully selected such that $(h,n)=1$ since $t$ can be arbitrarily large. Now consider the group $\Gamma=G\times H.$ It follows from $(h,n)=1$ that $S_c(\Gamma)=S_c(G)\cdot S_c(H)$ and the fitting length of $\Gamma$ is $l$. Hence $\frac{S_c(\Gamma)}{S_c(C_{nh})}=\frac{S_c(G)}{S_c(C_n)}\cdot\frac{S_c(H)}{S_c(C_h)}$ can be arbitrarily large as $t\mapsto\infty$ because $\frac{S_c(H)}{S_c(C_h)}$ is a fixed number. Hence, the group $\Gamma$ is as desired. 

A similar construction works while considering the derived length of a group. In other word, for any integer $d\geq 2,$   there exists a group with derived length of $d$ so that it does not satisfy the bound.

\end{document}